\documentclass{amsart}

\usepackage[utf8]{inputenc}

\usepackage[]{amsmath}
\usepackage{amsthm}
\usepackage{latexsym}
\usepackage{amssymb}
\usepackage{mathrsfs}
\usepackage{exscale}
\usepackage{textcomp}
\usepackage[all,ps,tips,tpic]{xy}
\usepackage{upgreek}
\usepackage{url}
\usepackage{booktabs}
\usepackage[final,pdftex,colorlinks=false,pdfborder={0 0 0}]{hyperref}
\usepackage{aliascnt}

\usepackage{tikz}
\usetikzlibrary{arrows,calc,matrix,patterns}

\newcommand{\Rat}{\mathbb{Q}}
\newcommand{\Nat}{\mathbb{N}}
\newcommand{\Int}{\mathbb{Z}}

\newcommand{\Top}{\mathit{Top}}

\newcommand{\defeq}{\mathrel{\mathop:}=}

\newcommand{\LST}[1]{\mathbf{\underline{#1}}}

\numberwithin{thmcounter}{section}
\newaliascnt{thmauto}{thmcounter}

\newaliascnt{defauto}{thmcounter}

\newaliascnt{exauto}{thmcounter}

\newaliascnt{lemauto}{thmcounter}

\newaliascnt{propauto}{thmcounter}

\newaliascnt{corauto}{thmcounter}

\newaliascnt{remauto}{thmcounter}

\theoremstyle{plain}
\newtheorem{thm}[thmauto]{Theorem}
\newtheorem*{meta}{Meta Question}

\newtheorem{prop}[propauto]{Proposition}
\newtheorem{cor}[corauto]{Corollary}
\newtheorem{thmA}{Theorem}

\theoremstyle{definition}
\newtheorem{definition}[defauto]{Definition}

\theoremstyle{remark}
\newtheorem{rem}[remauto]{Remark}
\newtheorem*{rem*}{Remark}

\let\originalleft\left
\let\originalright\right
\renewcommand{\left}{\mathopen{}\mathclose\bgroup\originalleft}
\renewcommand{\right}{\aftergroup\egroup\originalright}

\DeclareMathOperator{\Conf}{Conf}

\DeclareMathOperator{\Star}{Star}
\DeclareMathOperator{\HH}{H}
\DeclareMathOperator{\map}{map}
\DeclareMathOperator{\FI}{FI}

\usepackage{microtype}
\usepackage[draft]{fixme}
\usepackage{soul}

\usepackage{tikz}
\usepackage{pgfplots}
\usepgfplotslibrary{polar}

\definecolor{light-gray}{gray}{0.75}

\newcommand{\particleNr}[2]{
  \begin{scope}[shift={#1}]
    \node[fill=white,draw,circle,inner sep=1pt] at (0,0) {\small #2};
  \end{scope}
}

\title{Representation Stability for Configuration Spaces of Graphs}
\author{Daniel Lütgehetmann}
\address{Institut f\"ur Mathematik, Freie Universit\"at Berlin, Germany}
\email{daniel.luetgehetmann@fu-berlin.de}
\date{May 2019}

\begin{document}
\begin{abstract}
  We consider for two based graphs $G$ and $K$ the sequence of graphs $G_k$ given by the wedge sum
  of $G$ and $k$ copies of $K$.
  These graphs have an action of the symmetric group $\Sigma_k$ by permuting the $K$-summands.
  We show that the sequence of representations of the symmetric group $H_q(\Conf_n(G_\bullet);
  \Rat)$, the homology of the ordered configuration space of these spaces, is representation stable
  in the sense of Church and Farb.
  In the case where $G$ and $K$ are trees (with loops), we provide a similar result for glueing
  along arbitrary subtrees instead of the base point.
  Furthermore, we show that stabilization alway holds for $q=1$.
\end{abstract}
\maketitle

\section{Introduction}
For a topological space $X$ and a finite set $S$ we define the \emph{ordered configuration space of
$X$ with particles $S$} as
\[ \Conf_S(X) \defeq \left\{f\colon S\to X \text{ injective} \right\} \subset \map(S, X). \]
For $n\in\Nat$ we write $\LST{n}\defeq \{1, 2, \ldots, n\}$ and $\Conf_n(X)\defeq
\Conf_{\LST{n}}(X)$.

Let $G$ be a finite connected graph (by which we mean a 1-dimensional
CW-complex), then we are interested in the homology of $\Conf_n(G)$, the
\emph{ordered configuration space of $n$ particles in $G$}.
In \cite{Luetgehetmann14} we showed that at least one of the
$H^k(\Conf_\bullet(G);\Rat)$ for $k\ge 0$ cannot be representation stable.
In this paper we show that by stabilizing the \emph{graph} instead of the number
of particles we get representation stability.

Let $G_0$ be a finite graph and let $K_i\subset G_i$ for $1\le i\le \ell$ be pairs of finite graphs
such that each $K_i$ is also a subgraph of $G_0$.
Denote by $\underline{\Gamma}=\underline{\Gamma}_{G_0} \defeq \{(K_1, G_1), \ldots, (K_\ell,
  G_{\ell})\}$.
  Let $\mathbf{G}=\mathbf{G}_{\underline{\Gamma}}\colon \FI^{\times \ell} \to \Top$ be given
  by
  \[
    \mathbf{G}_{\underline{\Gamma}}(\LST{j_1}, \ldots, \LST{j_\ell}) \defeq
    G_0 \sqcup_{K_1} G_1^{\sqcup j_1} \cdots \sqcup_{K_\ell} G_\ell^{\sqcup j_\ell},
  \]
  i.e.\ by gluing the copies of the graphs $G_i$ to $G_0$ via the shared subgraph $K_i$.
  To define the images of morphisms, notice that each summand $G_i$ can be labeled by a number
  between 1 and $j_i$.
  For a map $\phi\colon\LST{j_i}\hookrightarrow\LST{j'_i}$ we define the induced map to send the
  summand with label $m\in\LST{j_i}$ to the summand with label $\phi(m)$ via the identity.

  \begin{meta}
    For which $G_i$, $K_i$, $q$ and abelian group $A$ is the
    $A[\FI^{\times\ell}]$-module
    \[
      \mathbf{H}^{A}_{q,n}\underline{\Gamma} \defeq
      H_q(\Conf_n(\mathbf{G}_{\underline{\Gamma}}); A)
    \]
    finitely generated?
  \end{meta}

  For $A=\Int$ we also write $\mathbf{H}^{\underline\Gamma}_{q,n}$.
  Part of this question asks how ``local'' the homology of configuration spaces
  of graphs is.
  In the case of trees with loops (i.e.\ graphs that can be constructed by
  glueing copies of $S^1$ to a finite tree) we proved a rather strong kind of
  locality in \cite{CheLue16} by describing an explicit generating system of
  products of 1-classes, and by refining the generating system we can prove in
  that case that the answer to the question above includes all trees with loops.

  \begin{thmA}\label{thm:tree-stabilization}
    If each of the graphs $G_i$ for $0\le i\le \ell$ is a tree with loops, then
    $\mathbf{H}^{\underline{\Gamma}}_{q,n}$ is finitely generated in degree $(\zeta, \zeta, \ldots,
    \zeta)$ for each $q,n\in\Nat$, where $\zeta = \zeta_{n,q} = \min\{2n, n+3q\}$.
  \end{thmA}

  For general graphs, we restrict ourselves to the first homology and recover an analogous statement:
  \begin{thmA}\label{thm:first-homology-group}
    For any choice of graphs $G_i$ and $K_i$ the $\FI$-module $\mathbf{H}^{\underline{\Gamma}}_{1,n}$ is
    finitely generated in degree $(n+3, n+3, \ldots, n+3)$ for each $n\in\Nat$.
  \end{thmA}

  Finally, wedging arbitrary graphs along a single vertex also gives
  representation stability.

  \begin{thmA}\label{thm:wedge-stabilization}
    If each graph $K_i$ is a single vertex $v\in G_0$ then
    $\mathbf{H}^{\underline{\Gamma}}_{q,n}$ is finitely generated in degree
    $(n+3, n+3, \ldots, n+3)$ for each $q,n\in\Nat$.
  \end{thmA}

  \begin{rem}
    If at least one of the $K_i$ has valence at least three inside $G_i$, then the proof of
    \autoref{thm:wedge-stabilization} actually shows that the $\FI$-module is generated in degree
    ${(n+1, \ldots, n+1)}$.
  \end{rem}

  \begin{cor}
    In the same situation as in the theorems above, choose any non-decreasing (component-wise) functor
    \[ F\colon \FI\to \FI^{\times\ell}, \]
    then the $\FI$-module $\mathbf{H}^{\underline{\Gamma}}_{q,n}\circ F$ is
    finitely generated.
    In particular, the sequence
    \[ H_q(\Conf_n(\mathbf{G}_{\underline{\Gamma}}\circ F); \Rat)\]
    is representation stable and therefore the dimension of the sequence of vector spaces is
    eventually polynomial.
  \end{cor}

  \begin{cor}\label{cor:thm-for-wedge}
    Let $G, K$ be finite graphs with base point and define
    \[ G_k \defeq G\vee \underbrace{K\vee\cdots\vee K}_{\text{$k$ times}}.\]
    Then the $\FI$-module $H_q(\Conf_n(G_\bullet))$ is finitely generated in degree $n+3$.
    In particular, the sequence $H_q(\Conf_n(G_\bullet);\Rat)$ is representation stable and therefore
    the dimension of the sequence of vector spaces is eventually polynomial in $k$.
  \end{cor}

  \begin{rem}
    The fact that the dimension of this sequence is \emph{bounded from above} by a polynomial can be
    seen more easily if $G$ and $K$ are trees:
    from \cite[Theorem 2.6, p.  8]{Ghrist01} we know that the rank of the first homology of star
    graphs is polynomial in the number of edges.
    This implies that the size of the generating set for $H_q(\Conf_n(G_k))$ described in
    \cite[Theorem 2, p.  2]{CheLue16} is polynomial in $k$, giving a polynomial upper bound for the
    rank. More recently, it was proved that for a very general class of graph
    stabilizations the corresponding sequence of homology of configuration
    spaces is representation stable (although without computing the explicit
    degree), see \cite[Theorem G, p. 4]{RaWh17}.
  \end{rem}

  \begin{rem}
    In a previous version of this paper, we studied stabilization of graphs
    along an interval or a circle.
    The proofs used results from \cite{CheLue16}, which had to be corrected.
    Without those results, it is much more tedious to compute the exact degree
    of stability, which is why we do not treat those types of stabilization in
    this paper as the statements were merely meant as a showcase on how to use
    the proof techniques for related problems.
    It is, however, easy to see that the homology of configuration spaces of
    such stabilizing graphs are representation stable of \emph{some} (unknown)
    degree, for a straightforward proof see e.g.\ \cite[Theorem G, p.
    4]{RaWh17}.
  \end{rem}

  \subsection{Acknowledgements}
  The author was supported by the Berlin Mathematical School and the SFB 647 “Space – Time –
  Matter” at Berlin.
  I would like to thank Peter Patzt and Elmar Vogt for fruitful discussions.

  \section{Representation stability and \texorpdfstring{$\FI^{\times\ell}$}{FI\textasciicircum
  l}-modules}
  \label{sec:fi-modules}
  In \cite{CF13}, Church and Farb introduced the concept of \emph{representation stability}.
  We now recall the concept in the case of the symmetric group $\Sigma_k$, for more details see
  \cite[Section 2.3, p. 19]{CF13}.

  Let $\{V_k\}_{k\in\Nat}$ be a sequence of $\Sigma_k$-representations over $\Rat$ with linear maps
  \[ \phi_k\colon V_k\to V_{k+1}\]
  which are homomorphisms of $\Rat\Sigma_k$-modules.
  Here we consider $V_{k+1}$ as $\Rat\Sigma_k$ module by the standard inclusion
  $\Sigma_k\hookrightarrow \Sigma_{k+1}$.

  To describe stability for such a sequence, we need to compare $\Sigma_k$-representations to
  $\Sigma_{k'}$-representations for $k' > k$.
  Recall that the irreducible representations of $\Sigma_k$ over the rational numbers are in one to
  one correspondence to partitions $\lambda$ of $k$.
  Given a partition $\lambda=(\lambda_1\ge \lambda_2\ge\cdots\ge \lambda_\ell\ge 0)$ of $k$, define
  for $k'-k\ge \lambda_1$ the irreducible $\Sigma_{k'}$-representation $V(\lambda)_{k'}$ to be the one
  corresponding to the partition $(k'-k, \lambda_1, \ldots, \lambda_\ell)$.
  Each irreducible representation can be written like this for a unique partition $\lambda$.
  For more details, see \cite[Section 2.1, p. 14]{CF13} and \cite{FH91}.

  \begin{definition}[{\cite[Definition 2.3, p.20]{CF13}}]
    The sequence $\{V_k\}$ is \emph{(uniformly) representation stable} if, for sufficiently large $k$,
    each of the following conditions holds.
    \begin{itemize}
      \item $\phi_k\colon V_k\to V_{k+1}$ is injective.
      \item The $\Rat\Sigma_{k+1}$ submodule generated by $\phi_k(V_k)$ is equal to $V_{k+1}$.
      \item Decompose each $V_k$ into irreducible representations
        \[ V_k = \bigoplus_\lambda c_{\lambda,k} V(\lambda)_k \]
        with multiplicities $0\le c_{\lambda,k}\le\infty$.
        Then there exists an $N\ge0$ such that for each $\lambda$, the multiplicity $c_{\lambda,k}$ is
        independent of $k\ge N$.
    \end{itemize}
  \end{definition}

  This reduces the description of the infinite sequence of $\Sigma_k$-representations to a finite
  calculation.

  \vspace{1em}

  In \cite{CEF15}, Church-Ellenberg-Farb introduced the notion of $\FI$-modules, which we now recall.
  Let $\FI$ be the category with objects all finite sets and morphisms all injective maps.
  We often consider the skeleton of this category given by the restriction to the finite sets
  $\LST{n}\defeq\{1,\ldots, n\}$ for $n\ge 0$.
  \begin{definition}
    Let $R$ be a commutative ring.
    An $R[\FI]$-module $V_\bullet$ is a functor
    \[ V_\bullet \colon \FI \to \mathrm{RMod}.\]
    It is said to be \emph{finitely generated in degree $\ell$} if there exists a finite set $X$ of
    elements in
    \[ \bigsqcup_{\substack{S\in\FI\\ |S|\le\ell}} V_S, \]
    such that the smallest sub-$\FI$-module containing all these elements is $V_\bullet$.
    Here, $|S|$ is the cardinality of $S$.
  \end{definition}
  For sequences of finite dimensional representations, the notion of a finitely generated $\FI$-module
  is a generalization of representation stable sequences by the following result:
  \begin{thm}[{\cite[Theorem 1.13, p. 8]{CEF15}}]
    An $\FI$-module $V_\bullet$ over a field of characteristic 0 is finitely generated if and only if
    the sequence $k\mapsto V_{\LST{k}}$ is representation stable and each $V_{\LST{k}}$ is
    finite dimensional.
  \end{thm}
  This result reduces the uniform decomposition of the representations $V_{\LST{k}}$ to finding a
  finite set of generators.
  Furthermore, Church-Ellenberg-Farb proved that the dimension of representation stable sequences
  grows polynomially:
  \begin{thm}[{\cite[Theorem 1.5, p. 4]{CEF15}}]
    Let $V_\bullet$ be an $\FI$-module over a field of characteristic 0.
    If $V_\bullet$ is finitely generated then the sequence of characters $\chi_{V_{\bullet}}$ is
    eventually polynomial. In particular, $\dim V_{\LST{k}}$ is eventually polynomial in $k$.
  \end{thm}

  In order to describe stabilization in multiple directions we look at the product category
  $\FI^{\times\ell}$ consisting of $\ell\ge1$ copies of the category $\FI$.
  An $\FI^{\times\ell}$-module is then a functor $\FI^{\times\ell}\to \mathrm{RMod}$, the notion of
  finite generation is defined analogously.

  To define such a module, it is sufficient to define it on the skeleton consisting of the objects
  $(\LST{j_1}, \ldots, \LST{j_\ell})$ for $j_i\in\Nat$ and the morphisms between them.
  In the introduction we defined $\mathbf{G}_{\underline{\Gamma}}$ for those objects.
  To define the images of morphisms, notice that each summand of $G_i$ can be labeled by a number
  between 1 and $j_i$.
  For a map $\phi\colon\LST{j_i}\hookrightarrow\LST{j'_i}$ we define the induced map to send the
  summand with label $m\in\LST{j_i}$ to the summand with label $\phi(m)$ via the identity.

  Clearly, if $V$ is a finitely generated $\FI^{\times\ell}$-module and
  $F\colon\FI\to\FI^{\times\ell}$ is any non-decreasing functor, then the $\FI$-module $F^*V\defeq
  V\circ F$ is finitely generated:
  since $F$ is non-decreasing, each component of $F$ is either eventually constant or unbounded.

  \section{Generators for the homology of configuration spaces of graphs}
  In this section, we give an overview over the results of \cite{CheLue16} that we need in this paper.

  \subsection{Configurations in trees with loops}
  We need the following definitions from the mentioned paper.
  \begin{definition}\label{def:product-of-cycles}
    A homology class $\sigma\in H_q(\Conf_n(G))$ is called the \emph{product of classes $\sigma_1\in
      H_{q_1}(\Conf_{T_1}(G_1))$ and $\sigma_2\in H_{q_2}(\Conf_{T_2}(G_2))$} if it is the image of
      $\sigma_1\otimes \sigma_2$ under the map
      \[
        H_q(\Conf_n(G_1\sqcup G_2)) \to H_q(\Conf_n(G))
      \]
      induced by an embedding $G_1\sqcup G_2\hookrightarrow G$.
      Analogously, we define iterated products.
    \end{definition}
    \begin{definition}\label{def:basic-classes}
      For $k\ge 3$ let $\Star_k$ be the star graph with $k$ leaves, let $\HH$ be the tree with two
      vertices of valence three and let $S^1$ be the circle with one vertex of valence 2.
      We call a class $\sigma\in H_1\left( \Conf_n(G) \right)$ \emph{basic} if there exists a piecewise
      linear embedding $\iota$ of $\HH$, $S^1$ or $\Star_k$ for some $k$ into $G$ such that $\sigma$ is
      in the image of the induced map $H_1(\Conf_n(\iota))$.
    \end{definition}

    We will use the following result:
    \begin{thm}[{\cite[Theorem D, p. 3]{CheLue16}}]\label{generators}
      Let $G$ be a finite graph.
      Then the first homology of $\Conf_n(G)$ is generated by basic classes.
      If $G$ is a tree with loops, then $H_*(\Conf_n(G))$ is free and
      generated by products of basic classes.
    \end{thm}
    \begin{rem}\label{rem:generators-homology-trees-h-graphs}
      We will also use that the embeddings of $\HH$ can be chosen such that they contain precisely two
      essential vertices, which can be arranged by splitting an $\HH$-graph
      containing $k$ such vertices into $k-1$ of them, each containing exactly
      two vertices.
      Also, after fixing those two vertices, we can choose the edges of the
      embedded $\HH$-graph arbitrarily: in the proof of the theorem above we
      only needed that the valence of the vertices is at least three.
      The cycles given by different choices of edges differ by cycles in the stars of the corresponding
      vertices.
      See \cite{CheLue16} for details.
    \end{rem}

  \subsection{The combinatorial model}\label{sec:the-combinatorial-model}
  In this paper, we will use the same combinatorial model of the configuration space of a graph as in \cite{CheLue16}.
  We only briefly sketch the construction, for more details see \cite{CheLue16} and
  \cite{Luetgehetmann14}.

  \begin{definition}[{Cube Complex, see \cite[Definition I.7.32]{bh10}}]
    A cube complex $K$ is the quotient of a disjoint union of cubes
    $X=\bigsqcup_{\lambda\in\Lambda}[0,1]^{k_\lambda}$ by an equivalence
    relation $\sim$ such that the quotient map $p\colon X\to X/\!\!\sim\
    = K$ maps each cube injectively into $K$ and we only identify faces
    of the same dimensions by an isometric homeomorphism.
  \end{definition}
  \begin{rem}
    The definition above differs slightly from the original definition
    by Bridson and Häfliger, in that it allows two cubes to be
    identified along more than one face.
  \end{rem}

  \begin{prop}[{\cite[Proposition 2.3, p. 4]{CheLue16}}]\label{prop:combinatorial-model}
    Let $G$ be a finite graph and $n\in\Nat$. Then $\Conf_{n}(G)$ deformation
    retracts to a finite cube complex of dimension $\min\{n, |V_{\ge2}| \}$,
    where $V_{\ge 2}$ is the set of vertices of $G$ of valence at least two.
  \end{prop}

  The basic idea of the combinatorial model is to keep all particles on
  any single edge equidistant at all times.  Moving one of the outmost
  particles from an edge to an empty essential vertex is then
  given by decreasing the distance of this particle from the vertex,
  while simultaneously increasing the distance between the particles on
  this edge.  Once the particle reaches the vertex, all remaining
  particles on the edge will be equidistant again.

  More formally, the 0-cubes of the combinatorial model are all those
  configurations where all particles in the interior of each edge cut
  the edge into pieces of equal length. A $k$-dimensional cube is given by
  choosing such a $0$-cell, $k$ distinct particles sitting on distinct
  vertices and for each of those particles an edge incident to the
  corresponding vertex.  The $i$-th dimension of the cube $[0,1]^k$ then
  corresponds to moving the $i$-th of those $k$ particles from their
  position on the vertex onto the edge, where at time zero the particle
  is on the vertex and at time 1 it is on the edge.  Remember that if
  there are already particles on the edge then they continuously squeeze
  together to make room for the new particle (it is also possible that
  two particles move onto the same edge from different sides).
  Such a choice of $k$ movements determines a $k$-cube if and only if we
  can realize the movements independently, namely if no two particles
  move towards the same vertex.  This describes the cube complex as a subspace
  of the configuration space.

  Each vertex can only be involved in one of those combinatorial movements at
  the same time, so the dimension of this cube complex is bounded above by the
  number of essential vertices.

  For an example, see Figure \ref{fig:combinatorial-model}.
  When we write $\Conf_n(X)$, we from now on mean the combinatorial model, unless
  stated otherwise.

  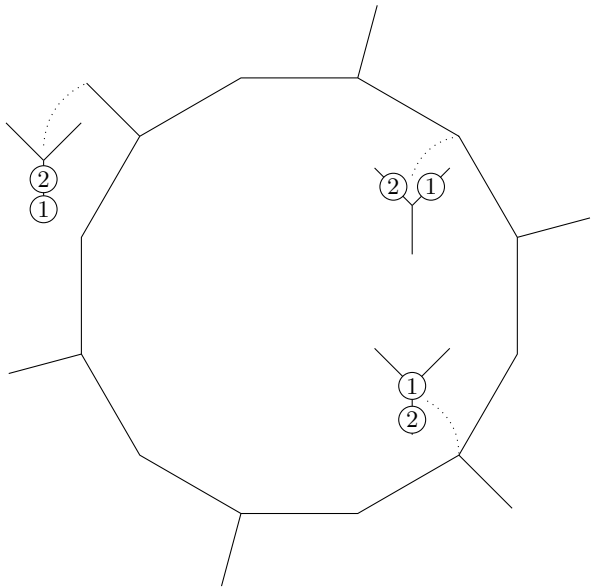
\begin{figure}[htpb]
    \centering
    \begin{tikzpicture}
      \newcommand{\Graph}{ \draw (0, 0) -- (0, -1.3); \draw (0, 0) --
        (1, 1); \draw (0, 0) -- (-1, 1); }

      \draw[-] ($ (0, 0) + (15: 3) $) -- ($ (0, 0) + (45: 3) $) -- ($
      (0, 0) + (75: 3) $) -- ($ (0, 0) + (105: 3) $) -- ($ (0, 0) +
      (135: 3) $) -- ($ (0, 0) + (165: 3) $) -- ($ (0, 0) + (195: 3)
      $) -- ($ (0, 0) + (225: 3) $) -- ($ (0, 0) + (255: 3) $) -- ($
      (0, 0) + (285: 3) $) -- ($ (0, 0) + (315: 3) $) -- ($ (0, 0) +
      (345: 3) $) -- cycle;

      \draw[-] ($ (0, 0) + (15: 3) $) -- ($ (0, 0) + (15: 4) $);
      \draw[-] ($ (0, 0) + (75: 3) $) -- ($ (0, 0) + (75: 4) $);
      \draw[-] ($ (0, 0) + (135: 3) $) -- ($ (0, 0) + (135: 4) $);
      \draw[-] ($ (0, 0) + (195: 3) $) -- ($ (0, 0) + (195: 4) $);
      \draw[-] ($ (0, 0) + (255: 3) $) -- ($ (0, 0) + (255: 4) $);
      \draw[-] ($ (0, 0) + (315: 3) $) -- ($ (0, 0) + (315: 4) $);

      \begin{scope}[shift={(1.5, -1.2)}, scale=0.5]
        \Graph
        \particleNr{(0, 0)}{1};
        \particleNr{(0, -0.9)}{2};
      \end{scope}

      \begin{scope}[shift={(1.5, 1.2)}, scale=0.5]
        \Graph
        \particleNr{(0.5, 0.5)}{1};
        \particleNr{(-0.5, 0.5)}{2};
      \end{scope}

      \begin{scope}[shift={(-3.4, 1.8)}, scale=0.5]
        \Graph
        \particleNr{(0, -1.3)}{1};
        \particleNr{(0, -.5)}{2};
      \end{scope}

      \draw[dotted] (-3.4, 2) to[bend left] ($ (0, 0) + (135: 4) $);
      \draw[dotted] (1.7, -1.4) to[bend left] ($ (0, 0) + (315: 3) $);
      \draw[dotted] (1.5, 1.6) to[bend left] ($ (0, 0) + (45: 3) $);
    \end{tikzpicture}

    \caption{The combinatorial model of $\Conf_2(Y)$. Each edge
      corresponds to the movement of a single particle from the
      essential vertex onto one of the three edges. Moving along the
      embedded circle the two particles move alternatingly onto the
      edge that is not occupied by the other particle.}
      \label{fig:combinatorial-model}
  \end{figure}

  \section{Proof of the main theorems}
  We first prove \autoref{cor:thm-for-wedge} for star graphs by hand.

  \begin{prop}\label{prop:thm-for-star}
    \autoref{cor:thm-for-wedge} is true for $G$ the point and $K$ the interval $[0,1]$ with 0 as base
    point.
  \end{prop}
  \begin{rem}\label{rem:stabilization-2n}
    For $n=2$ the argument presented below is easily modified to show that $H_1(\Conf_2(G_\bullet))$
    is generated in degree $n+2=4$.
    Since $n+3\le 2n$ for $n>2$, this shows that $H_1(\Conf_n(G_\bullet))$ is generated in degree
    $2n$, which will be used in the proof of \autoref{thm:tree-stabilization}.
  \end{rem}
  \begin{proof}[{Proof of \autoref{prop:thm-for-star}}]
    The combinatorial model of this configuration space is a graph (c.f.\ \cite[Theorem 2.6, p.
    8]{Ghrist01}, \cite{Luetgehetmann14}), so we only need to consider 1-cycles.
    Choose any subgraph $\Star_3 \subset \Star_k$.
    Let $C$ be a 1-cycle, then the claim is that we can write $C$ as a sum of cycles where
    each particle uses at most one edge outside of $\Star_3$.

    Let $x$ be a particle and choose a 0-cube $\nu$ of $C$ where $x$ sits on the vertex of the star.
    If this does not exist, then $x$ is fixed and therefore uses at most one
    edge and we have nothing to prove for this particle.
    Now move along a path $\gamma$ of 1-cubes in the cycle until $x$ sits on the vertex again and
    there exists a continuation such that the next 1-cube would move $x$ onto an edge of
    $\Star_k-\Star_3$ for the \emph{second time}.
    We denote the corresponding terminal 0-cube of $\gamma$ by $\nu'$.
    Now choose the following path $\gamma'$ back to $\nu$, during which $x$ always stays in $\Star_3$:
    move $x$ onto an edge $e_1$ of $\Star_3$ and keep it there.
    Follow $\gamma$ back ignoring the movement of $x$ and using the connectedness of the configuration
    space of $\Star_3$ to move $x$ out of the way if other particles need to move along $e_1$.
    Finally, move $x$ back to the vertex.

    This decomposes $C$ into two cycles:
    the cycle $\gamma\gamma'$ and $C$ with $\gamma$ replaced by $\gamma'^{-1}$.
    In the first of those two cycles the particle $x$ only visits one edge not in $\Star_3$.
    Continuing this process, we eventually exhaust all 1-cubes of $C$ and get a sum decomposition of $C$
    where in each summand $x$ visits only $\Star_3$ and at most one additional edge.

    Since we did not increase the number of edges outside of $\Star_3$ visited by any other particle,
    we can repeat this for every $x$ and get a sum decomposition of $C$ of the required form.

    Consequently, for each $N\ge n+3$ we can generate $H_1(\Conf_n(G_N))$ by cycles such that each one
    of them is supported in some subgraph $\Star_{n+3}\hookrightarrow G_N$.
    Therefore, the $\Int\Sigma_N$-span of the image of the map
    \[ H_1(\Conf_n(G_{n+3}))\to H_1(\Conf_n(G_N)) \]
    is the whole module and the $\FI$-module $H_1(\Conf_n(G_\bullet); \Int)$ is finitely generated in
    degree $n+3$.
  \end{proof}

  \begin{proof}[{Proof of \autoref{thm:tree-stabilization}}]
    Let $n>1$ and let $(k_1, \ldots, k_\ell)$ be such that each $k_i$ is at least $\zeta=\min\{2n,
    n+3q\}$.
    By \autoref{generators}, the homology of $\Conf_n(\mathbf{G}(k_1, \ldots, k_\ell))$ is generated
    by products of basic cycles.
    By \autoref{rem:generators-homology-trees-h-graphs}, we can assume that the embedded $\HH$-graphs
    contain exactly two vertices because $k_i > 3$ and therefore the valence of all internal vertices
    is at least three.
    In the following, we will say that a particle meets a copy of some $G_i$ if it moves into the part
    not contained in $G_0$, namely the part $G_i-K_i$.

    Each \textbf{$\mathbf{\HH}$-class} meets at most 3 copies of each $G_i$ by
    \autoref{rem:generators-homology-trees-h-graphs}.
    Since each $\HH$-class consists of $m\ge 2$ particles, we have $3\le \zeta_{m,1}$.

    Each \textbf{star class} with $m$ particles can be written as a linear combination of generators
    such that each is using only $\zeta_{m,1}$ different edges by \autoref{prop:thm-for-star} and
    \autoref{rem:stabilization-2n}.
    Therefore, each summand visits at most $\zeta_{m,1}$ distinct copies of each of the $G_i$.

    The only embedded copies of $S^1$ are given by self loops at one of the vertices, so each
    \textbf{$\mathbf{S^1}$-class} meets at most one of the copies of one of the $G_i$.

    Each of the non-moving particles meets at most one of the copies.
    Hence, we can generate the whole homology by classes which each meet at most
    \[
      \zeta_{m_1,1}+\cdots+ \zeta_{m_q,1}+(n-m_1-\cdots-m_q)\le \min\{2n, n+3q\}=\zeta_{n,q}
    \]
    different copies of each of the $G_i$.
    This implies that the $\Int[\Sigma_{k_1}\times\cdots\times\Sigma_{k_\ell}]$-span of the image of
    \[ H_q\left( \Conf_n(\mathbf{G}(\zeta,\ldots,\zeta)) \right) \to H_q\left(
    \Conf_n(\mathbf{G}(k_1,\ldots,k_\ell)) \right) \]
    is the whole module, finishing the proof.
  \end{proof}

  \begin{proof}[{Proof of \autoref{thm:first-homology-group}}]
    By \autoref{generators}, the homology group
    $\mathbf{H}^{\underline{\Gamma}}_{1,n}$ is generated by basic classes.
    Fix $(k_1, \ldots, k_\ell)$ such that $k_i\ge n+3$ for each $i$.
    Each $\HH$-class meets at most three copies of each $G_i$ by
    \autoref{rem:generators-homology-trees-h-graphs}.
    Star classes involving $k$ particles can be written as sums of other star classes, each meeting at
    most $k+3$ copies of each of the $G_i$ by \autoref{prop:thm-for-star}.
    Every $S^1$-class can be written as a sum of $S^1$-classes such that each of them meets at most
    two copies of each of the $G_i$:
    choose a spanning tree for each connected component of $\mathbf{G}(1,\ldots, 1)\subset
    \mathbf{G}(k_1, \ldots, k_\ell)$ and extend it to spanning trees for the connected components of
    $\mathbf{G}(k_1,\ldots, k_\ell)$.
    The inclusion $\mathbf{G}(1,\ldots,1)\hookrightarrow \mathbf{G}(k_1,\ldots, k_\ell)$ is a
    $\pi_0$-isomorphism, so this construction ensures that this forest restricted to the union of
    $\mathbf{G}(1,\ldots,1)$ and a copy of one of the $G_i$ still gives a spanning forest.
    The cycles corresponding to the edges outside of that spanning forest thus stay inside this copy
    and $\mathbf{G}(1,\ldots,1)$, so they meet at most two copies of each $G_i$.

    The non-moving particles meet at most one copy each, so each class can be written as a sum of
    classes meeting at most $n+3$ copies of each of the $G_i$.
  \end{proof}

  For the proof of \autoref{thm:wedge-stabilization} we need the following definition:

  \begin{definition}[{\cite[Definition 3.6, p. 41]{Luetgehetmann17}}]
    For finite sets $T \subset S$, a finite graph G and a subset $K\subset G$ define
    \[
      \Conf_{S, T}((G, K)) = \{f\colon S \hookrightarrow G \,|\, f(T)\subset K\}\subset \Conf_S(G).
    \]
  \end{definition}

  Additionally, we use the following result, whose proof we sketch here for the readers convenience.

  \begin{prop}[{\cite[Proposition 5.6, p. 8]{LueRM18}}]\label{prop:adding-particles}
    For each $n\ge m$ and each graph $G$ with at least one essential
    vertex there exists a map
    \[
    \Conf_m(G) \to \Conf_n(G)
    \]
    which composed with the forgetful map
    \[
    \Conf_n(G) \to \Conf_m(G)
    \]
    is homotopic to the identity.  In particular, we have that
    $H_*(\Conf_m(G))$ is a direct summand of $H_*(\Conf_n(G))$.
  \end{prop}

  \begin{proof}[{Sketch of proof for the case $n=m+1$}]
    We define the map between the combinatorial models as follows:
    Choose an essential vertex $v$ and three edges $e_1, e_2, e_3$ incident to $v$.

    For each $k$-cube in the combinatorial model of $\Conf_m(G)$ where no particle moves from $e_1$
    towards $v$ simply add particle $n=m+1$ onto $e_1$ between $v$ and all other particles on $e_1$.

    Given a cube where one particle $p$ moves from $e_1$ towards $v$ we
    consider the following sequence of movements: move particle $n$ via $v$ onto $e_2$, move $p$ via
    $v$ onto $e_3$, move $n$ back onto $e_1$ and finally move $p$ onto $v$. These movements are
    independent of the movements of the other particles in the chosen cube, so we can replace the
    movement of $p$ with this sequence.
    This defines a union of cells in the combinatorial model of $\Conf_n(G)$, and we define our map to
    stretch the cube we started with onto this strip of cells, see
    \autoref{fig:strip-of-cells}.  It is straightforward to check that
    this gives a continuous map with the desired properties.\qedhere

    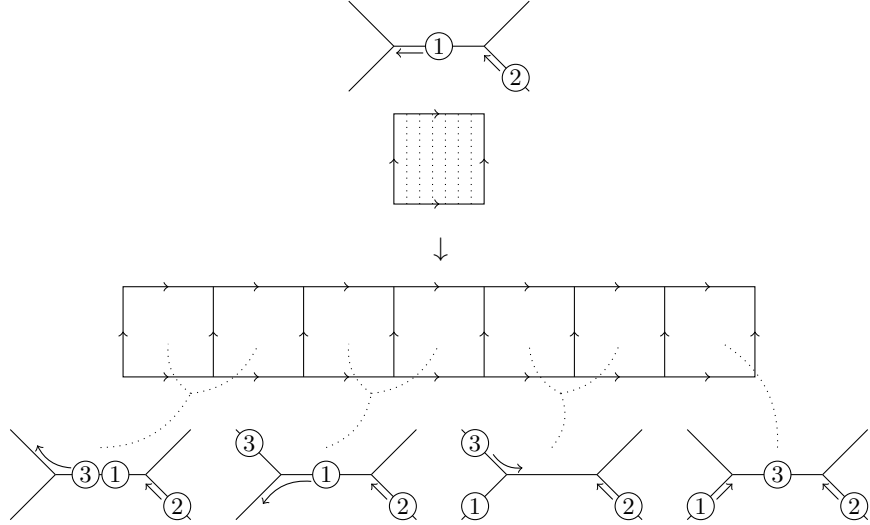
\begin{figure}[htpb]
      \centering
      \begin{tikzpicture}
        \newcommand{\Graph}{ \draw (-1, 0) -- (1, 0); \draw (-2, -1) --
          (-1, 0); \draw (-2, 1) -- (-1, 0); \draw (2, -1) -- (1, 0);
          \draw (2, 1) -- (1, 0); }

        \begin{scope}[shift={(0, 5)}, scale=0.6]
          \Graph
          \particleNr{(0, 0)}{1};
          \particleNr{(1.7, -0.7)}{2}; \draw[->] (-0.35, -0.15) to
          (-0.95, -0.15); \draw[->] (1.35, -0.55) to (1.0, -0.2);
        \end{scope}

        \begin{scope}[shift={(0, 3.5)}, scale=0.6]
          \draw (-1, -1) -- (1, -1) -- (1, 1) -- (-1, 1) -- cycle;
          \draw[->] (-.005, -1) -- (.005, -1); \draw[->] (-.005, 1) --
          (.005, 1); \draw[->] (-1, -.005) -- (-1, .005); \draw[->] (1,
          -.005) -- (1, .005); \draw[dotted] (-0.7143, -1) -- (-0.7143,
          1); \draw[dotted] (-0.4286, -1) -- (-0.4286, 1); \draw[dotted]
          (-0.1429, -1) -- (-0.1429, 1); \draw[dotted] (0.1428, -1) --
          (0.1428, 1); \draw[dotted] (0.4285, -1) -- (0.4285, 1);
          \draw[dotted] (0.7142, -1) -- (0.7142, 1);
        \end{scope}

        \node[rotate=-90] at (0, 2.3) {$\to$};

        \begin{scope}[shift={(-1.8, 1.2)}, scale=0.6]
          \draw (-4, -1) -- (10, -1) -- (10, 1) -- (-4, 1) -- cycle;

          \draw[->] (-3.005, -1) -- (-2.995, -1); \draw[->] (-3.005, 1)
          -- (-2.995, 1); \draw[->] (-1.005, -1) -- (-0.995, -1);
          \draw[->] (-1.005, 1) -- (-0.995, 1); \draw[<-] (1.005, -1) --
          (0.995, -1); \draw[<-] (1.005, 1) -- (0.995, 1); \draw[<-]
          (3.005, -1) -- (2.995, -1); \draw[<-] (3.005, 1) -- (2.995,
          1); \draw[<-] (5.005, -1) -- (4.995, -1); \draw[<-] (5.005, 1)
          -- (4.995, 1); \draw[<-] (7.005, -1) -- (6.995, -1); \draw[<-]
          (7.005, 1) -- (6.995, 1); \draw[<-] (9.005, -1) -- (8.995,
          -1); \draw[<-] (9.005, 1) -- (8.995, 1);

          \draw[->] (-2, -.005) -- (-2, .005); \draw[->] (-4, -.005) --
          (-4, .005); \draw[->] (2, -.005) -- (2, .005); \draw[->] (0,
          -.005) -- (0, .005); \draw[->] (4, -.005) -- (4, .005);
          \draw[->] (6, -.005) -- (6, .005); \draw[->] (8, -.005) -- (8,
          .005); \draw[->] (10, -.005) -- (10, .005); \draw (-2, -1) --
          (-2, 1); \draw (0, -1) -- (0, 1); \draw (2, -1) -- (2, 1);
          \draw (4, -1) -- (4, 1); \draw (6, -1) -- (6, 1); \draw (8,
          -1) -- (8, 1);
        \end{scope}

        \begin{scope}[shift={(0, -0.7)}, scale=0.6]
          \begin{scope}[shift={(-7.5, 0)}]
            \Graph
            \particleNr{(0.33, 0)}{1};
            \particleNr{(-0.33, 0)}{3};
            \particleNr{(1.7, -0.7)}{2}; \draw[->, bend left] (-0.65,
            0.15) to (-1.4, 0.65); \draw[->] (1.35, -0.55) to (1.0,
            -0.2);

            \draw[dotted, bend right] (0, 0.6) to (2, 1.8);
            \draw[dotted, bend left] (2, 1.8) to (1.5, 2.9);
            \draw[dotted, bend right] (2, 1.8) to (3.5, 2.9);
          \end{scope}
          \begin{scope}[shift={(-2.5, 0)}]
            \Graph
            \particleNr{(0, 0)}{1};
            \particleNr{(-1.7, 0.7)}{3};
            \particleNr{(1.7, -0.7)}{2}; \draw[->, bend right] (-0.33,
            -0.15) to (-1.4, -0.65); \draw[->] (1.35, -0.55) to (1.0,
            -0.2);

            \draw[dotted, bend right] (0, 0.6) to (1, 1.8);
            \draw[dotted, bend left] (1, 1.8) to (0.5, 2.9);
            \draw[dotted, bend right] (1, 1.8) to (2.5, 2.9);
          \end{scope}
          \begin{scope}[shift={(2.5, 0)}]
            \Graph
            \particleNr{(-1.7, -0.7)}{1};
            \particleNr{(-1.7, 0.7)}{3};
            \particleNr{(1.7, -0.7)}{2}; \draw[<-, bend left] (-0.65,
            0.15) to (-1.3, 0.45); \draw[->] (1.35, -0.55) to (1.0,
            -0.2);

            \draw[dotted, bend right] (0, 0.6) to (0.2, 1.8);
            \draw[dotted, bend left] (0.2, 1.8) to (-0.5, 2.9);
            \draw[dotted, bend right] (0.2, 1.8) to (1.5, 2.9);
          \end{scope}
          \begin{scope}[shift={(7.5, 0)}]
            \Graph
            \particleNr{(-1.7, -0.7)}{1};
            \particleNr{(0, 0)}{3};
            \particleNr{(1.7, -0.7)}{2}; \draw[->] (-1.35, -0.55) to
            (-1.0, -0.2); \draw[->] (1.35, -0.55) to (1.0, -0.2);

            \draw[dotted, bend right] (0, 0.6) to (-1.2, 2.9);
          \end{scope}
        \end{scope}
      \end{tikzpicture}
      \caption{Replacing a 2-cell by a strip of 2-cells to construct a
        map $\Conf_2(G)\to \Conf_3(G)$. The vertical direction
        in the cubes corresponds to the movement of particle 2, the
        seven small rectangles above are stretched to the seven squares
        below.}
      \label{fig:strip-of-cells}
    \end{figure}
  \end{proof}

  \begin{proof}[{Proof of \autoref{thm:wedge-stabilization}}]
    From now on, we will assume that each $G_i$ has at least one edge, i.e.\ $G_i\neq K_i$, and we
    will write $G=\mathbf{G}(k_1, k_2, \ldots, k_{\ell})$ for some fixed numbers $k_1, \ldots,
    k_{\ell} \ge n+3$.
    Choosing three edges incident to $v$, the subgraph of $G$ consisting of $v$ and these three
    edges intersects at most three $G_i^r$ outside of $G_0\defeq \mathbf{G}(0, \ldots, 0)$ and is
    called the base graph of $G$ and is denoted by $G_\mathrm{base}$.

    \vspace{1em}

    For this proof only we say that a cellular chain
    \[
      X\in C_q(\Conf_n(G)) = C^\mathrm{cell}_q(\Conf_n(G))
    \]
    representing a homology class is bounded of degree $0\le m\le n$ if it is in the image of the map
    \[
      C_q\left( \Conf_{\LST{n}, \LST{m} }( G, L_m )\right) \to C_q\left(\Conf_n(G)\right),
    \]
    induced by the inclusion, where $L_m$ is some subgraph of $G$ given by the union of
    $G_\mathrm{base}$ and $m$ different $G^r_i$.
    A homology class in $H_q(\Conf_n(G))$ is bounded of degree $m$ if it can be represented by a
    cellular chain with this property, and the homology group $H_q(\Conf_n(G))$ itself is said to be
    bounded of degree $m$ if it is generated by all classes bounded of this degree.

    We will proceed by proving that $H_q(\Conf_n(G))$ is bounded of degrees $m=0, \ldots, n$ via
    induction on $m$.
    For the base case $m=0$ there is nothing to prove, so let's assume that $m>0$ and that
    $H_q(\Conf_n(G))$ is bounded of degree $m - 1$.
    It now suffices to show that each class bounded of degree $m - 1$ can be written as a sum of
    classes which are bounded of degree $m$.

    Let $X$ be a $q$-dimensional cellular chain representing such a class, then by definition the
    particles $1, 2, \ldots, m - 1$ meet only a graph $L_{m-1}$ given by the union of the base graph
    and $m-1$ different $G^r_i$.
    Let particle $m$ be on some subgraph $G^{r_0}_{i_0} - \{v\}$ of $G-L_{m - 1}$ for some cell of
    $X$, and define $L_m\defeq L_{m - 1} \cup G^{r_0}_{i_0}$.
    If such a cell would not exist, then $X$ itself is bounded of degree $m$ and there would be
    nothing to show.

    Otherwise, consider $X$ as a cellular chain of $\Conf_{\LST{n}, \LST{m-1}}(G, L_{m -
    1})$.
    Let $X_{L_m}$ be the part of the linear combination $X$ given by all cells where particle $m$
    does not leave $L_m$.
    The boundary $Y=\partial X_{L_m}$ has the particle $m$ fixed on the vertex $v$ and by
    $\partial Y=\partial^2 X_{L_m}=0$ we see that $Y$ represents a class in $H_{q-1}(\Conf_{\LST{n},
    \LST{m}}(G, L_{m-1}))$.
    As $X_{L_m}$ bounds $Y$, this is the trivial homology class, and we claim that there exists a
    cellular chain
    \[ X_{L_m}'\in C_q(\Conf_{\LST{n}, \LST{m}}(G, L_{m-1})) \subset
    C_q(\Conf_{\LST{n}, \LST{m-1}}(G, L_{m-1}))\]
    that bounds $Y$ without the particle $m$ meeting the complement of the base graph
    $G_\mathrm{base}$.

    The class $[Y]$ is trivial in $H_{q-1}(\Conf_{\LST{n}, \LST{m}}(G, L_{m-1}))$, so its
    image under the projection
    \[
      H_{q-1}\left(\Conf_{\LST{n}, \LST{m}}(G, L_{m-1})\right) \to
      H_{q-1}\left(\Conf_{\LST{n} - \{m\}, \LST{m - 1}}(G, L_{m - 1})\right)
    \]
    is also trivial.
    Choosing an explicit bounding chain for the image of $Y$ under the corresponding map of cellular
    chain complexes, we can use \autoref{prop:adding-particles} to put the particle $m$ onto an
    edge $e$ of $G_\mathrm{base}$ incident to $v$, giving a chain in $C_q\left(\Conf_{\LST{n},
    \LST{m}}(G, L_{m - 1})\right)$ bounding a chain $\tilde{Y}$ that is
    almost equal to $Y$: the particle $m$ is fixed on the edge $e$ instead of
    $v$ (in $Y$ no particle moves towards $v$).
    Therefore, $\tilde{Y}$ is homologous in $\Conf_{\LST{n}, \LST{m}}(G, L_{m-1})$ to $Y$, giving
    the required chain $X_{L_m}'$.

    Now $X_{L_m} - X_{L_m}'$ is a chain where $m$ does not leave $L_m$, so it represents a homology
    class that is bounded of degree $m$, and in $X-(X_{L_m}-X_{L_m}')$ (the rest of $X$) the
    particle $m$ does not meet $G^{r_0}_{i_0}-\{v\}$.
    Repeating this procedure eventually leads to a decomposition of $X$ into chains that are bounded
    of degree $m$.

    \vspace{1em}

    By induction, the homology is generated by classes that are bounded of degree $n$, and since
    those classes meet at most $n+3$ different $G^r_i$ (or $n+1$ if $G_\mathrm{base}$ can be choosen to be contained
    in a single $G^r_i$), this concludes the proof.
  \end{proof}

\bibliographystyle{alpha}
\bibliography{conf-graph-stability}
\end{document}